\documentclass[a4paper,english,10pt]{article}
\usepackage{amsmath,amsthm,amssymb,amsxtra}
\usepackage{mathrsfs}
\usepackage{pdfsync}
\usepackage{babel} 
\usepackage{latexsym} 
\usepackage{textcomp} 
\usepackage{bm}  
\usepackage{colortbl}
\allowdisplaybreaks

\newcommand*{\scrF}{\ensuremath{\mathscr{F}}} 
\newcommand*{\caC}{\ensuremath{\mathcal{C}}}	
\newcommand*{\caF}{\ensuremath{\mathcal{F}}}	
\newcommand*{\caN}{\ensuremath{\mathcal{N}}}	
\newcommand*{\caS}{\ensuremath{\mathcal{S}}}	

\newcommand*{\N}{\mathbb{N}}									
\newcommand*{\R}{\mathbb{R}}									
\newcommand*{\Rd}{{\mathbb{R}^d}}							

\newcommand*{\eps}{\varepsilon}								

\newcommand*{\E}{\mathbb{E}}									
\renewcommand*{\P}{\mathbb{P}}								

\numberwithin{equation}{section}
\newtheorem{lemma}{Lemma}[section]
\newtheorem{theorem}[lemma]{Theorem}

\newtheorem{remark}[lemma]{Remark}

\allowdisplaybreaks

\begin{document}

\begin{titlepage}
\vskip 1cm

\begin{center}
{\Large\bf Absolute continuity for SPDEs with irregular fundamental solution}
\medskip

by\\
\vspace{14mm}

\begin{tabular}{l@{\hspace{10mm}}l@{\hspace{10mm}}l}
{\sc Marta Sanz-Sol\'e}$\,^{(\ast)}$ & and &{\sc Andr\'e S{\"u}\ss}$\,^{(\ast)}$\\
{\small marta.sanz@ub.edu}        & &{\small andre.suess@ub.edu}\\
{\small http://www.ub.edu/plie/Sanz-Sole}\\
\end{tabular}
\begin{center}
{\small Facultat de Matem\`atiques}\\
{\small Universitat de Barcelona } \\
{\small Gran Via de les Corts Catalanes, 585} \\
{\small E-08007 Barcelona, Spain} \\
\end{center}

\vskip 1cm

\end{center}

\vskip 2cm

\noindent{\bf Abstract.} 
For the class of stochastic partial differential equations studied in \cite{conusdalang}, we prove the existence of density of the 
probability law of the solution at a given point $(t,x)$, and that the density belongs to some Besov space. The proof relies on the method developed in \cite{debusscheromito}.
The result can be applied to the solution of the stochastic wave equation with multiplicative noise, Lipschitz coefficients and  any spatial dimension $d\ge 1$, and also to the heat equation.
This provides an extension of the results proved in \cite{sanzsuess1}.

 \medskip

\smallskip


\vskip 2cm

\noindent
\footnotesize
{\begin{itemize} \item[$^{(\ast)}$] Supported by the grant MTM 2012-31192 from the \textit{Direcci\'on General de
Investigaci\'on Cient\'{\i}fica y T\'ecnica, Ministerio de Econom\'{\i}a y Competitividad, Spain.}
\end{itemize}}

\end{titlepage}
\newpage

\section {Introduction}
The seminal article \cite{malliavin1} begins with a criterion for the absolute continuity with respect to the Lebesgue measure for nonnegative finite measures on $\R^m$. Using tools of harmonic analysis, it is proved that if $\kappa$ is such a measure and there exist a constant $c$ such that for every $\Phi$ with compact support, we have $\left\vert\int \partial_k \Phi d \kappa\right\vert \le c\Vert \Phi\Vert_{\infty}$, for $1\le k\le m$, then $\kappa(dx)= k(x)dx$ and $k\in L^1$. By iteration, it is possible to strengthen this criterion and obtain the existence of an infinitely differentiable density. In the same article, Malliavin sets up the grounds of a stochastic calculus of variation with the purpose to be able to apply this criterion to the probability law of Gaussian functionals. It was coined as {\it Malliavin Calculus}. The book \cite{malliavin2} contains an extensive list of references on applications of this calculus which is still continuing to grow. Among them, the study of the law of random field solutions to stochastic partial differential equations, in the sequel referred to as SPDEs (see \cite{sanzbook} for an introduction). 

The use of Malliavin calculus for the analysis of densities requires some {\it regularity} properties that are not met by all SPDEs and neither by stochastic differential equations with non-smooth coefficients. This problem has motivated the search for alternatives to Malliavin's criterion, quite similar in spirit, but giving weaker conclusions, applicable to cases that exhibit irregularities. For the sake of brevity, we only mention the references \cite{fournierprintems,debusschefournier,debusscheromito} where the approach used in this paper is developed, and the recent related article \cite{ballycaramellino}.

Throughout the paper we consider the setting of \cite{conusdalang,sanzsuess1}. More explicitly, we deal with an SPDE 
$ Lu(t,x) = b(u(t,x)) + \sigma(u(t,x))\dot{F}(t,x)$, with suitable initial conditions, that we express in its mild formulation as 
\begin{align}\label{eq:SPDE2}
  u(t,x) =	& \int_0^t\int_\Rd \Lambda(t-s,x-y)\sigma(u(s,y))M(ds,dy) \notag\\
						& + \int_0^t\int_\Rd \Lambda(t-s,x-y)b(u(s,y))dyds,
\end{align}
$(t,x)\in[0,T]\times \R^d$. Here $\Lambda$ denotes the fundamental solution to $Lu=0$ and $M$ is the martingale measure derived from a random noise $F$ white in time
and with a stationary covariance measure in space. The spectral measure of the covariance (its inverse Fourier transform) will be denoted by $\mu$.

Let $\{u(t,x), (t,x)\in[0,T]\times \R^d\}$ be the random field solution to \eqref{eq:SPDE2}. Assume that the function $\sigma$ in \eqref{eq:SPDE2} is constant, and fix
$(t,x)\in(0,T]\times \R^d$. In \cite{sanzsuess1}, using Malliavin Calculus, it is proved that the probability law of $u(t,x)$ has a density. The purpose of this article is to extend this result allowing $\sigma$ to be a nonlinear Lipschitz continuous function. Moreover, we prove that the density belongs to
some Besov space.

In \cite{sanzsuess1}, the restriction on $\sigma$ is forced by the method of the proof. Indeed, in the examples where the fundamental solution $\Lambda$ is a nonnegative distribution, $\sigma$ and $b$ are differentiable, and $\sigma$ is bounded away from zero, we can prove that the Malliavin matrix is invertible. However, for more general $\Lambda$, for example the fundamental solution to the wave equation in dimension $d\ge 4$, this does not seem to be feasible, except for constant $\sigma$. In contrast, the method of \cite{debusscheromito}, based on Lemma \ref{lem:existencedensity} of Section \ref{sec:result}, can be successfully applied, and also the regularity of the coefficients $\sigma$ and $b$ can be relaxed.

The paper is structured in the following way. In Section \ref{sec:result} we prove the main result on existence of density, and find the Besov space that contains this density.  In Section \ref{sec:example}, we study the example of stochastic wave equations in any spatial dimension $d\geq 1$, the interesting and novel case being $d\ge 4$ (see \cite{quersanz1}, \cite{quersanz2} for related results). We consider two cases of spectral measures $\mu$: with densities given by a Riesz kernel, and finite measures. We also provide a comment about the stochastic heat equation. The existence of a density for this equation is well-known, see \cite{mcms}; however, with this different approach we can allow for less smooth coefficients.

We end this section by fixing some notation. Throughout the article we write $C$ for any positive constant, which may change from line to line. The set of Schwartz functions on $\Rd$ is denoted by $\caS(\Rd)$, $\caS'_r(\Rd)$ is the set of tempered distributions with rapid decrease, and  $\caF$  the Fourier transform operator on $\Rd$. We denote by $\{ \scrF_t, t\in[0,T]\}$ the filtration generated by the martingale measure $\{M_t, t\in[0,T]\}$.




\section{Statement and proof of the main result}\label{sec:result}
The objective of this section is to prove Theorem \ref{thm:existencedensity}. We begin by introducing a first set of relevant assumptions:
\smallskip

\noindent{\bf (A1)} 
$t\mapsto\Lambda(t)$ is a deterministic function with values in $\mathcal{S}'_r(\Rd)$; the mapping $(t,\xi)\mapsto\caF\Lambda(t)(\xi)$ is measurable and
  \begin{align*}
  &\int_0^T \sup_{\eta\in\Rd} \int_{\Rd}  |\caF\Lambda(s)(\xi+\eta)|^2 \mu(d\xi) ds< \infty,\\
  &\int_0^T \sup_{\eta\in\Rd} |\caF\Lambda(s)(\eta)|^2 ds < \infty. 
  \end{align*}
\noindent{\bf (A2)} Let $\phi$ denote a nonnegative function in $\mathcal{C}^\infty_0(\Rd)$, with support included in the unit ball of $\Rd$, satisfying $\int_\Rd \phi(x)dx=1$. For all such $\phi$ and all $0\leq a\leq b\leq T$, we have
  \[ \int_a^b (\Lambda(s)\ast\phi)(x) ds \in \mathcal{S}(\Rd) \]
  and
  \[ \int_{\Rd}\int_a^b |(\Lambda(s)\ast\phi)(x)| ds dx < \infty. \]
\noindent{\bf (A3)} $t\mapsto\caF\Lambda(t)$ is as in {\bf (A1)}  and
	\begin{align*}
  &\lim_{h\downarrow0} \int_0^T \sup_{\eta\in\Rd}\int_{\Rd}\sup_{s<r<s+h}|\caF\Lambda(r)(\xi+\eta)-\caF\Lambda(s)(\xi+\eta)|^2\mu(d\xi)\ ds = 0,\\
  &\lim_{h\downarrow0} \int_0^T \sup_{\eta\in\Rd}\sup_{s<r<s+h}|\caF\Lambda(r)(\eta)-\caF\Lambda(s)(\eta)|^2  ds = 0. 
  \end{align*}

Under the assumptions {\bf (A1)} and either {\bf (A2)} or {\bf (A3)}, \cite[Theorem 3.1]{conusdalang} assures that the integrals in \eqref{eq:SPDE2} are well-defined and the equation has a unique random field solution. Moreover, for all $t\in[0,T]$, $u(t,x)$ has the same distribution as $u(t,0)$, for all $x\in\Rd$. The solution $\{u(t,x)), (t,x)\in[0,T]\times\Rd\}$ is $L^2$-continuous and has uniformly bounded second moments. We also recall the following estimates:
\begin{align}\label{eq:isometry1}
  \E\bigg[\bigg(\int_0^t\int_\Rd & \Lambda(t-s,x-y)\sigma(u(s,y))M(ds,dy)\bigg)^2\bigg] \notag \\ 
		& \leq \int_0^t \E\big[\sigma(u(s,0))^2\big]\sup_{\eta\in\Rd}\int_\Rd |\caF\Lambda(t-s)(\xi+\eta)|^2\mu(d\xi)ds,
\end{align}
and
\begin{align}\label{eq:isometry2}  
  \E\bigg[\bigg(\int_0^t\int_\Rd & \Lambda(t-s,x-y)b(u(s,y))dyds\bigg)^2\bigg] \notag\\
									& \leq \int_0^t \E\big[b(u(s,0))^2\big]\sup_{\eta\in\Rd} |\caF\Lambda(t-s)(\eta)|^2ds
\end{align}
(see \cite{conusdalang} for the details).

The proof of the existence of density for the law of the solution requires the following second set of assumptions. 
\smallskip

\noindent {\bf (A4)} There exists $C,\delta>0$ such that $\E[(u(t,0)-u(s,0))^2]\leq C|t-s|^\delta$, for all $s,t\in[0,T]$.

\noindent{\bf (A5)} There exists $\sigma_0>0$ such that $\inf_{x\in\R} |\sigma(x)|=\sigma_0$. 

\noindent{\bf (A6)}  There exist positive constants $C,\gamma, \gamma_1, \gamma_2>0$ and $t_0\in(0,T]$ such that 
\begin{align}
Ct^{\gamma} \leq g(t)&:=\int_0^t \int_{\Rd} |\caF\Lambda(s)(\xi)|^2 \mu(d\xi) ds,\ \text{for all }t\in[0,t_0], \label{A8a}\\
 g_1(t)&:=\int_0^t \sup_{\eta\in\Rd} \int_\Rd |\caF \Lambda(s)(\xi+\eta)|^2 \mu(d\xi) d s \leq Ct^{\gamma_1}, \label{A8b}\\
  g_2(t)&:=\int_0^t \sup_{\eta\in\Rd} |\caF \Lambda(s)(\eta)|^2  d s \leq Ct^{\gamma_2}, \label{A8c}\\
  \notag
\end{align}
for all $t\in[0,T].$

The assumption {\bf (A5)} (strong ellipticity) appears frequently when studying the absolute continuity of probability measures induced by solutions to SDEs and SPDEs, while {\bf (A6)} (or similar ones) has been usually required to prove regularity properties of the density.

This is the main result of the paper.

\begin{theorem}\label{thm:existencedensity}
	Fix $(t,x)\in(0,T]\times\Rd$. We assume that the coefficients $\sigma$ and $b$ are Lipschitz continuous functions. Moreover, suppose that {\bf (A1)}, either {\bf (A2)} or {\bf (A3)}, {\bf (A4)}, {\bf (A5)} and {\bf (A6)} hold, and that
	\begin{equation}\label{eq:condition}
	  \bar\gamma := \frac{\min\{\gamma_1,\gamma_2\}+\delta}{\gamma} > 1.
	\end{equation}
	Then, the probability law of $u(t,x)$ is absolutely continuous and its density belongs to all Besov spaces $B_{1,\infty}^s$, with
	\begin{equation}\label{eq:conclusion}
		s < 1 - \bar\gamma^{-1}.
	\end{equation}
\end{theorem}

The spaces $B_{1,\infty}^s$, $s>0$,  can be defined as follows.
Let $f:\Rd\to\R$. For $x,h\in\Rd$ set $(\Delta^1_hf)(x)=f(x+h)-f(x)$. Then, for any $n\in\N$, $n\geq 2$, let
\[ (\Delta_h^nf)(x) = \big(\Delta^1_h(\Delta^{n-1}_hf)\big)(x) = \sum_{j=0}^n (-1)^{n-j}\binom{n}{j}f(x+jh). \]
For any $0<s<n$, we define the norm
\[ \|f\|_{B^s_{1,\infty}} = \|f\|_{L^1} + \sup_{|h|\leq 1} |h|^{-s}\|\Delta_h^nf\|_{L^1}. \]
It can be proved that for two distinct $n,n'>s$ the norms obtained using $n$ or $n'$ are equivalent. 
Then we define $B^s_{1,\infty}$ to be the set of $L^1$-functions with $\|f\|_{B^s_{1,\infty}}<\infty$.
We refer the reader to \cite{triebel} for more details.

The proof of Theorem \ref{thm:existencedensity} is based on the following lemma from \cite{debusschefournier} (based on \cite{debusscheromito}). In the following, we denote by $\caC^\alpha_b$ the set of bounded H\"older continuous functions of degree $\alpha$.

\begin{lemma}\label{lem:existencedensity}
Let $\kappa$ be a finite nonnegative measure. Assume that there exist $0<\alpha\leq a<1$, $n\in\N$ and a constant $C_n$ such that for all $\phi\in\caC^\alpha_b$, and all $h\in\R$ with $|h|\leq1$,
\begin{equation}
	\bigg|\int_\R \Delta_h^n\phi(y)\kappa(d y)\bigg|\leq C_n\|\phi\|_{\caC^\alpha_b}|h|^a.
	\label{eq:existencedensity}
\end{equation}
Then $\kappa$ has a density with respect to the Lebesgue measure, and this density belongs to the Besov space $B^{a-\alpha}_{1,\infty}(\R)$.
\end{lemma}

We will apply this lemma to $\kappa= \P\circ u(t,x)^{-1}$, where $u$ is the solution to the SPDE \eqref{eq:SPDE2} at some fixed point $(t,x)\in(0,T]\times\Rd$. We prove the claim in Theorem \ref{thm:existencedensity} for $x=0$. Since the probability distribution of $u(t,x)$ does not depend on $x\in\Rd$, this yields the result.

Note that in this application the constant $C_n$ in \eqref{eq:existencedensity} may depend on $n$, on the elements defining the SPDE \eqref{eq:SPDE2} and on the assumptions. In particular, on $\sigma,\sigma_0,b,\beta,T,d,\gamma,\gamma_1,\gamma_2$ or the total mass $|\mu|$ if the measure $\mu$ is finite, and also on $t\in(0,T]$. 

In order to apply Lemma \ref{lem:existencedensity}, we rely on the three next lemmas.

\begin{lemma}\label{lem:gradientestimate}
The density $\varphi$ of a one-dimensional normal distribution $\caN(0,\sigma^2)$ satisfies
\[ \big\| \varphi^{(n)}\big\|_{L^1} = C_n\big(\sigma^2\big)^{-n/2}, \]
for all $n\in\N$, where $\varphi^{(n)}(y)= \frac{d^n\varphi}{dy^n}(y)$, 
$y\in\R$.
\begin{proof}
Let $H_n$ denote the $n$-th Hermite polynomial. It is well-known that
\[ \varphi^{(n)}(y) = n!\bigg(\frac{-1}{(2\sigma^2)^{1/2}}\bigg)^n H_n\bigg(\frac{y}{(2\sigma^2)^{1/2}}\bigg)\frac{1}{(2\sigma^2)^{1/2}}\exp\bigg(-\frac{y^2}{2\sigma^2}\bigg), \]
With a change of variables we obtain
\[ \|\varphi^{(n)}\|_{L^1} = \frac{n!}{(2\sigma^2)^{n/2}}\int_\R \big|H_n(y)\big|\exp(-y^2)d y.\]
Since the last integral is finite, we have the result.
\end{proof}
\end{lemma}

For $0<\eps<t$, define
\begin{align}\label{eq:addueps} 
	u^\eps(t,0) &=  \int_0^{t-\eps}\int_\Rd \Lambda(t-s,-y)\sigma(u(s,y))M(ds,dy)\notag\\
	&+ \int_0^{t-\eps}\int_\Rd \Lambda(t-s,-y)b(u(s,y))dyds \notag\\
			 & + \sigma(u(t-\eps,0))\int_{t-\eps}^t\int_\Rd \Lambda(t-s,-y)M(d s,d y)\notag\\
			 & + b(u(t-\eps,0))\int_{t-\eps}^t\int_\Rd \Lambda(t-s,-y)dyds. 
\end{align}

The following lemma gives a first bound for the expected values of the iterated differences of functions of the solution to the SPDE \eqref{eq:SPDE2}.

\begin{lemma}\label{lem:Deltahn}
	Under the assumptions in Theorem \ref{thm:existencedensity}, we have for every $\alpha\in(0,1)$, $\phi\in\caC^\alpha_b$, $h\in\R$, $t\in[0,T]$ and $0<\eps<t$,
	\begin{equation}\label{eq:Deltahn}
		\big|\E\big[\Delta_h^n\phi(u(t,0))\big]\big| \leq C_n\|\phi\|_{\caC^\alpha_b}\Big(|h|^{n} g(\eps)^{-n/2} + \big(\E\big[|u^\eps(t,0)-u(t,0)|^2\big]\big)^{\alpha/2}\Big),
	\end{equation}
	where $u^\eps$ and $g$ are defined in \eqref{eq:addueps} and \eqref{A8a}, respectively.
\begin{proof}
The left-hand side of \eqref{eq:Deltahn} satisfies
\[ \big|\E\big[\Delta_h^n\phi(u(t,0))\big]\big| \leq I_1(h,n,\phi,\eps,t) + I_2(h,n,\phi,\eps,t), \]
where
\begin{align*}
	I_1(h,n,\phi,\eps,t) & := \big|\E\big[\Delta_h^n\phi(u(t,0)) - \Delta_h^n\phi(u^{\eps}(t,0))\big]\big|, \\
	I_2(h,n,\phi,\eps,t) & := \big|\E\big[\Delta_h^n\phi(u^{\eps}(t,0))\big]\big|.
\end{align*}
For the first term, the property $\|\Delta_h^n\phi\|_{\caC^\alpha_b}\leq C_n\|\phi\|_{\caC^\alpha_b}$, the spatial stationarity of the solution and H\"older's inequality yield
\begin{align}\label{eq:addI_1}
	I_1(h,n,\phi,\eps,t)
	& \leq C_n\|\phi\|_{\caC^\alpha_b}\E\big[|u^\eps(t,0)-u(t,0)|^\alpha\big] \notag \\
	& \leq C_n\|\phi\|_{\caC^\alpha_b}\big(\E\big[|u^\eps(t,0)-u(t,0)|^2\big]\big)^{\alpha/2}.
\end{align}
For the study of the term $I_2(h,n,\phi,\eps,t)$, we consider the decomposition
\begin{equation}\label{eq:addueps2}
  u^\eps(t,0) = U_t^\eps + \sigma(u(t-\eps,0))\int_{t-\eps}^t\int_\Rd \Lambda(t-s,-y)M(d s,d y),
\end{equation}
where
\begin{align*}
	U_t^\eps = 	& \int_0^{t-\eps}\int_\Rd \Lambda(t-s,-y)\sigma(u(s,y))M(ds,dy) 	\\
							& + \int_0^{t-\eps}\int_\Rd \Lambda(t-s,-y)b(u(s,y))dyds					\\
							& + b(u(t-\eps,0))\int_{t-\eps}^t\int_\Rd \Lambda(t-s,-y)dyds.
\end{align*}
Notice that $U^\eps_t$ is  $\scrF_{t-\eps}$-measurable, and conditionally to $\scrF_{t-\eps}$,  
\[ V_t^\eps:=\sigma(u(t-\eps,0))\int_0^t\int_\Rd \Lambda(t-s,-y)M(ds,dy)\]
is a Gaussian random variable with zero mean and independent of $U^\eps_t$. The conditional variance of $V_t^\eps$ is computed as follows:
\begin{align}\label{eq:variancefepst}
	\sigma^2_\Lambda(\eps)
	& := \E\bigg[\bigg(\sigma(u(t-\eps,0))\int_{t-\eps}^t\int_\Rd \Lambda(t-s,-y)M(d s,d y)\bigg)^2\bigg|\scrF_{t-\eps}\bigg] \notag\\
	& = \sigma(u(t-\eps,0))^2\E\bigg[\bigg(\int_{t-\eps}^t\int_\Rd \Lambda(t-s,-y)M(d s,d y)\bigg)^2\bigg] \notag\\
	& = \sigma(u(t-\eps,0))^2\int_0^\eps\int_\Rd |\caF\Lambda(s)(\xi)|^2\mu(d\xi)d s, \notag\\
	& \geq \sigma_0^2 g(\eps). 
\end{align}
where in the last step we have used {\bf (A5)}. Therefore the conditional law of $V_t^\eps$ with respect to $\scrF_{t-\eps}$ has a $\caC^\infty_b$-density, which we denote by $\varphi_{t,\Lambda,\eps}$. 

For any $f\in\caC^m$, we have
\begin{equation}\label{eq:boundDeltahn}
	\|\Delta_h^nf\|_{L^1(\R)} \leq C_n|h|^n\|f^{(n)}\|_{L^1(\R)}.
\end{equation}
Indeed, this inequality holds since 
$\Delta_h^nf(x) = \int_0^n l_n(v) f^{(n)}(x+hv)h^nd v$,
for some bounded function $l_n$ which is independent of $f$.

Therefore, by conditioning with respect to $\scrF_{t-\eps}$, and applying a discrete integration by parts, \eqref{eq:boundDeltahn}, Lemma \ref{lem:gradientestimate} and \eqref{eq:variancefepst}, we obtain
\begin{align*}
  I_2(h,n,\phi,\eps,t) 
  & = \bigg|\E\bigg[\int_\R \Delta_h^n\phi(U_t^\eps + y)\varphi_{t,\Lambda,\eps}(y)d y\bigg]\bigg| \\
  & = \bigg|\E\bigg[\int_\R \phi(U_t^\eps + y)\Delta_{-h}^n\varphi_{t,\Lambda,\eps}(y)d y\bigg]\bigg| \\
  & \leq \|\phi\|_\infty \int_\R\big|\Delta_{-h}^n\varphi_{t,\Lambda,\eps}(y)\big|d y \\
  & \leq C_n\|\phi\|_\infty |h|^n \|\varphi^{(n)}_{t,\Lambda,\eps}\|_{L^1(\R)} \\
  & = C_n\|\phi\|_\infty |h|^n (\sigma^2_\Lambda(\eps))^{-n/2} \\
  & \leq C_{n,\sigma_0}\|\phi\|_{\infty} |h|^n g(\eps)^{-n/2},
\end{align*}
which together with \eqref{eq:addI_1} yields \eqref{eq:Deltahn}, since $\|\phi\|_{\infty}\leq \|\phi\|_{\caC^\alpha_b}$.
\end{proof}	
\end{lemma}

\begin{lemma}\label{lem:differenceuueps}
Under the assumptions in Theorem \ref{thm:existencedensity} we have for all $t\in[0,T]$ and all $\eps\in(0,t)$,
\begin{equation}
	\E\big[(u(t,0)-u^\eps(t,0))^2\big] \leq C\eps^{\delta}\big(g_1(\eps)+g_2(\eps)\big),
	\label{eq:difference2}
\end{equation}
with $\delta$ as in {\bf (A4)}, and $g_1$, $g_2$ are defined in \eqref{A8b}, \eqref{A8c}, respectively.
\begin{proof}
Using \eqref{eq:isometry1}, \eqref{eq:isometry2} and the Lipschitz continuity of $\sigma$ and $b$, we have
\begin{align*}
  & \E\big[(u(t,0)-u^\eps(t,0))^2\big] \\
  & \leq 2\E\bigg[\bigg(\int_{t-\eps}^t\int_\Rd \Lambda(t-s,-y)\big(\sigma(u(s,y))-\sigma(u(t-\eps,0))\big)M(d s,d y)\bigg)^2\bigg]  \\
  & \phantom{\leq} + 2\E\bigg[\bigg(\int_{t-\eps}^t\int_\Rd \Lambda(t-s,-y)\big(b(u(s,y))-b(u(t-\eps,0))\big)d yd s\bigg)^2\bigg] \\
  & \leq \int_{t-\eps}^t \E\big[|\sigma(u(s,0))-\sigma(u(t-\eps,0))|^2\big] \sup_{\eta\in\Rd}\int_\Rd |\caF\Lambda(t-s)(\xi+\eta)|^2\mu(d\xi)ds \\
  & \phantom{\leq} + \int_{t-\eps}^t \E\big[|b(u(s,0))-b(u(t-\eps,0))|^2\big] \sup_{\eta\in\Rd} |\caF\Lambda(t-s)(\eta)|^2 ds \\
  & \leq C\sup_{s\in[t-\eps,t]} \E\big[(u(s,0)-u(t-\eps,0))^2\big]\big(g_1(\eps)+g_2(\eps)\big).
\end{align*}
Together with {\bf (A4)}, this implies \eqref{eq:difference2}. 
\end{proof}
\end{lemma}

We are now in a position to show Theorem \ref{thm:existencedensity}. 

\begin{proof}[Proof of Theorem \ref{thm:existencedensity}]
Fix $t\in(0,T]$, $x=0$, and let $\kappa = \P\circ u(t,0)^{-1}$. For all $h\in\R$ such that $|h|\leq1$ and all $\phi\in\caC^\alpha_b$ with $\alpha\in(0,1)$, we set
\[ I_{t,h} = \int_\R \Delta_h^n \phi(y)\kappa(d y) = \E\big[\Delta_h^n \phi(u(t,0))\big]. \]
Applying Lemmas \ref{lem:Deltahn} and \ref{lem:differenceuueps}, and {\bf (A6)}, we get 
\begin{align}
  \vert I_{t,h}\vert &\leq C_n\|\phi\|_{\caC^\alpha_b}\Big(|h|^n g(\eps)^{-n/2} + \big(\eps^{\delta}(g_1(\eps)+g_2(\eps))\big)^{\alpha/2}\Big)\nonumber\\
 &\leq C_{n}\|\phi\|_{\caC^\alpha_b}\Big(|h|^n \eps^{-\frac{\gamma n}{2}} + \eps^{\frac{\alpha(\gamma_1+\delta)}{2}} + \eps^{\frac{\alpha(\gamma_2+\delta)}{2}}\Big).\label{eq:addmaster2}
\end{align}
In the last inequality, we have used that  
$(x+y)^{\alpha/2}\leq 2^{\frac{\alpha}{2}-1}\big(x^{\alpha/2}+y^{\alpha/2}\big) \leq x^{\alpha/2}+y^{\alpha/2}$. Hence, the constant $C_n$ does not depend on $\alpha$.

Set $\eps=\tfrac{t}{2}|h|^{\frac{\rho}{\gamma}}$ with $\rho\in(0,2)$ to be selected later. Notice that for all $h\in[-1,1]$ we have $0<\eps<t$.
For this choice of $\eps$ and $n$ sufficiently large, 
\begin{equation*}
	\vert I_{t,h}\vert \leq C_{n,t}\|\phi\|_{\caC^\alpha_b}\Big(|h|^{\frac{\alpha\rho}{2}\frac{\min(\gamma_1,\gamma_2)+\delta}{\gamma}} \Big).
\end{equation*}
Fix $\rho\in\left(\frac{2}{\bar\gamma}, 2\right)$. Since $\bar\gamma>1$, one can choose $\alpha\in(0,1)$ satisfying $\frac{\alpha\rho \bar\gamma}{2}<1$.
Summarizing, we have proved that, for $n$ sufficiently large, there exists $\alpha\in(0,1)$ and $\rho\in\left(\frac{2}{\bar\gamma}, 2\right)$ satisfying
\begin{equation*}
\vert I_{t,h}\vert \leq C_{n,t}\|\phi\|_{\caC^\alpha_b}\vert h\vert^{\frac{\alpha\rho\bar\gamma}{2}},
\end{equation*}
and $0<\alpha<\frac{\alpha\rho\bar\gamma}{2}<1$.
Therefore, from Lemma \ref{lem:existencedensity} it follows that $\P\circ u(t,x)^{-1}$ has a density $g_{t,x}$ with respect to the Lebesgue measure, and $g_{t,x}\in B_{1,\infty}^{\frac{\alpha\rho\bar\gamma}{2}-\alpha}$.

We end the proof by determining the best degree of the Besov space. For this, we have to find $\max_{\alpha, \rho} \alpha\left(\frac{\rho\bar\gamma}{2}-1\right)$ with the restrictions $\alpha\in(0,1)$, $\rho\in(0,2)$ and $\frac{\alpha\rho\bar\gamma}{2}<1$. Using Lagrange's method we can prove that the unique optimal parameters for $\alpha$ and $\rho$ are $\bar\gamma^{-1}$ and $2$, respectively. 
Thus,
$g_{t,x}\in B_{1,\infty}^{s}$, with $s\in (0,1-\bar\gamma^{-1})$. This finishes the proof of the Theorem.

\end{proof}


\section{Examples}
\label{sec:example}
In this section we consider mainly the stochastic wave equation in any spatial dimension $d\ge 1$, which is studied in \cite[Section 4]{conusdalang}. In the last part, we will give some remarks on the heat equation (see Remark \ref{rem:heat}) which complement known results.

Let us consider \eqref{eq:SPDE2} where  $\Lambda$ is the fundamental solution to the wave equation, whose Fourier transform is 
\begin{equation}\label{eq:Lambdawaveeqaution}
	\caF\Lambda(t)(\xi) = \frac{\sin(t|\xi|)}{|\xi|}.
\end{equation}
We will assume that the spatial covariance of the noise $F$ is a Riesz kernel of parameter $\beta\in(0,2\wedge d)$ and therefore, that
the spectral measure $\mu$ is 
\begin{equation}
\label{eq:riesz2}
	\mu(d\xi) = \vert\xi\vert^{-d+\beta} d\xi, \ \xi\in\R^d.
\end{equation}
This kernel is fairly common in the literature on SPDEs with spatially homogeneous covariance. For example, it is a particular case of those considered  in \cite{dalang--sanz-sole09}. 

\begin{theorem}\label{thm:existencedensitywave}
Consider the SPDE \eqref{eq:SPDE2} where $\Lambda$ is the fundamental solution to the wave equation and the spectral measure of the noise $F$ is given by \eqref{eq:riesz2}, with $\beta\in[0,2\wedge d)$. Suppose that {\bf (A5)} is satisfied and that the coefficients $\sigma$ and $b$ are Lipschitz continuous. Fix $(t,x)\in(0,T]\times\Rd$. Then, the probability law of $u(t,x)$ is absolutely continuous with respect to the Lebesgue measure on $\R$ and its density belongs to all Besov spaces $B_{1,\infty}^s$ with $s\in\left(0,(2-\beta)/(5-2\beta)\right)$.
\end{theorem}

To prove this theorem it suffices to check that the assumptions of Theorem \ref{thm:existencedensity} hold with $\delta=2-\beta$, $\gamma=\gamma_1=3-\beta$ and $\gamma_2=3$. The conditions {\bf (A1)} and {\bf (A2)} have already been proved in \cite{conusdalang}. The remaining conditions are established in the next result.

\begin{lemma}\label{lem:differenceuuepswave}
The hypotheses are as in Theorem \ref{thm:existencedensitywave}. Then {\bf (A4)}, {\bf (A6)} hold with $\delta=2-\beta$, $\gamma=\gamma_1=3-\beta$ and $\gamma_2=3$. 
\begin{proof}
We basically follow the same method as in \cite[Proposition 4.1]{dalang--sanz-sole09}. Using \eqref{eq:isometry1} and \eqref{eq:isometry2}, we have 
\begin{align*}
	\E\big[(u(t,x)-u(s,x))^2\big] 
	\leq	& \sup_{r\in[0,T]} \E\big[\sigma(u(r,0))^2\big]\big(I_1(s,t)+I_2(s,t)\big) \\
				& + \sup_{r\in[0,T]} \E\big[b(u(r,0))^2\big]\big(I_3(s,t)+I_4(s,t)\big),  
\end{align*}
where
\begin{align*}	
	I_1(s,t) & = \int_0^s \sup_{\eta\in\Rd}\int_\Rd |\caF\Lambda(t-r)(\xi+\eta) - \caF\Lambda(s-r)(\xi+\eta)|^2 \mu(d\xi)d r, \\
	I_2(s,t) & = \int_s^t \sup_{\eta\in\Rd}\int_\Rd |\caF\Lambda(t-r)(\xi+\eta)|^2 \mu(d\xi)d r, \\
	I_3(s,t) & = \int_0^s \sup_{\eta\in\Rd} |\caF\Lambda(t-r)(\eta) - \caF\Lambda(s-r)(\eta)|^2 d r, \\
	I_4(s,t) & = \int_s^t \sup_{\eta\in\Rd} |\caF\Lambda(t-r)(\eta)|^2 d r.
\end{align*}
Using the identity $\sin x - \sin y = 2\sin\tfrac{x-y}{2}\cos\tfrac{x+y}{2}$, and the changes of variable, $\zeta\mapsto \tfrac{t-s}{2}(\xi+\eta)$ and $\xi\mapsto \zeta-\eta$, we get 
\begin{align*}
	& I_1(s,t) \\
	& = \int_0^s \sup_{\eta\in\Rd}\int_\Rd \bigg|\frac{\sin\big((t-r)|\xi+\eta|\big)}{|\xi+\eta|} - \frac{\sin\big((s-r)|\xi+\eta|\big)}{|\xi+\eta|}\bigg|^2|\xi|^{-d+\beta} d\xi dr \\
	& \leq 4\int_0^s \sup_{\eta\in\Rd}\int_\Rd \frac{1}{|\xi+\eta|^2|\xi|^{d-\beta}}\sin^2\bigg(\frac{(t-s)|\xi+\eta|}{2}\bigg) d\xi d r \\
	& = \frac{4s}{2^{d+2-\beta}}(t-s)^{2-\beta} \sup_{\eta\in\Rd}\int_\Rd \frac{\sin^2(|\zeta|)}{|\zeta|^2 |\zeta-2(t-s)^{-1}\eta|^{d-\beta}} d\zeta \\
	& = \frac{s}{2^{d-\beta}}(t-s)^{2-\beta} \sup_{\eta\in\Rd}\int_\Rd \frac{\sin^2(|\zeta|)}{|\zeta|^2|\zeta-\eta|^{d-\beta}} d\zeta \\
	& \leq \frac{T}{2^{2-\beta}} (t-s)^{2-\beta} \sup_{\eta\in\Rd} \int_\Rd \frac{\sin^2(|\xi+\eta|)}{|\xi+\eta|^2|\xi|^{d-\beta}} d\xi \\	
	& = C (t-s)^{2-\beta}.
\end{align*}
The last step hold because the integral is finite. Indeed, as in \cite[Lemma 6.1]{sanzbook} we can show that $\frac{\sin^2(|\xi+\eta|)}{|\xi+\eta|^2} \leq C\frac{1}{1+|\xi+\eta|^2}$. Therefore, 
\[ \sup_{\eta\in\Rd} \int_\Rd \frac{\sin^2(|\xi+\eta|)}{|\xi+\eta|^2|\xi|^{d-\beta}} d\xi \leq C\sup_{\eta\in\Rd}\int_\Rd \frac{1}{(1+|\xi+\eta|^2)|\xi|^{d-\beta}} d\xi, \]
and the integral on the right-hand side is finite (uniformly in $\eta$) if and only if $\beta\in(0,2\wedge d)$. 

For the term $I_2(s,t)$ we argue quite similarly as for $I_1(s,t)$. Using the change of variables $\zeta\mapsto(t-u)(\xi+\eta)$ and $\xi\mapsto \zeta-\eta$, we obtain
\begin{equation}\label{eq:addI2}
	I_2(s,t) = \int_s^t (t-u)^{2-\beta} \sup_{\eta\in\Rd} \int_\Rd \frac{\sin^2(|\xi+\eta|)}{|\xi+\eta|^2 |\xi|^{d-\beta}} d\xi d u = C (t-s)^{3-\beta},  
\end{equation}
 For $I_3(s,t)$ we use the Lipschitz continuity of the $\sin$ function to get
\[ I_3(s,t) \leq \int_0^s (t-s)^2 du \leq T(t-s)^2. \]
Finally, for $I_4(s,t)$ we use the property $|\sin(x)|\leq x$, for all $x\geq 0$, to obtain
\begin{equation}\label{eq:addI4}
	I_4(s,t) \leq \int_s^t \sup_{\eta\in\Rd} (t-u)^2 d u \leq C (t-s)^3.  
\end{equation}
Hence, we have proved that {\bf (A4)} holds with $\delta=2-\beta$.

Finally, we have to check {\bf (A6)}. With the change of variable $\eta = s\xi$, we clearly have
\begin{equation*}
 g(t)= \int_0^t\int_\Rd \frac{\sin^2(s|\xi|)}{|\xi|^{d+2-\beta}}d\xi ds 
    = c t^{3-\beta}\int_\Rd \frac{\sin^2(|\eta|)}{|\eta|^{d+2-\beta}}d\eta.
\end{equation*}
Thus \eqref{A8a} holds with $\gamma=3-\beta$. Notice that $g_1(t)=I_2(0,t)$ and $g_2(t)=I_4(0,t)$. Therefore \eqref{A8b}, \eqref{A8c} hold with $\gamma_1=3-\beta$ and $\gamma_2=3$, respectively.

The proof of the Lemma is complete.
\end{proof}
\end{lemma}

\begin{remark}
\label{rem:mufinite}
Assume that the spectral measure $\mu$ is finite. With the same hypotheses as in Theorem \ref{thm:existencedensitywave} we can prove the existence of density for the law of $u(t,x)$, for $(t,x)\in(0,T]\times \R^d$, and that this density belongs to the spaces $B_{1,\infty}^s$,
with $s\in(0,2/5)$. 
\end{remark}
Indeed, referring to the notations in Lemma \ref{lem:differenceuuepswave}, for a finite measure $\mu$ we have $I_1(s,t)\le C I_3(s,t)$
and $I_2(s,t)\le C I_4(s,t)$, $0\le s<t\le T$. Hence, {\bf (A4)} holds with $\delta=2$. 

Also $g_1(t)\le C g_2(t)$, which yields \eqref{A8b}, \eqref{A8c} with $\gamma_1=\gamma_2=3$. Moreover, since 
\begin{equation*}
\int_0^t\frac{\sin^2(s|\xi|)}{|\xi|^{2}}ds \ge C(t\wedge t^3)\frac{1}{1+|\xi|^2},
\end{equation*}
(see e.g. \cite[Lemma 6.1]{sanzbook}), it follows that \eqref{A8a} holds with $\gamma=3$.

\begin{remark}
\label{rem:heat}
Consider the SPDE \eqref{eq:SPDE2} where $\Lambda$ is the fundamental solution to the heat equation with $d\ge 1$. We assume that the spectral measure of the noise $F$ is either given by \eqref{eq:riesz2}, with $\beta\in[0,2\wedge d)$ or finite. Suppose that {\bf (A5)} is satisfied and that the coefficients $\sigma$ and $b$ are Lipschitz continuous. Fix $(t,x)\in(0,T]\times\Rd$. Then, the probability law of $u(t,x)$ is absolutely continuous with respect to the Lebesgue measure on $\R$ and its density belongs to all Besov spaces $B_{1,\infty}^s$ with $s\in\left(0,\frac{1}{2}\right)$.
\end{remark}  

Let $\mu$ be given by \eqref{eq:riesz2}. The case $\mu$ finite is left to the reader. Under the standing assumptions, there exists a random field solution to \eqref{eq:SPDE2} (see \cite{dalang}). In \cite{sanzsarra} it is proved that {\bf (A4)} holds with $\delta=1-\beta/2$. Hence, going through the proof of Theorem \ref{thm:existencedensity} we see that we only need to check hypotheses {\bf (A6)}. 

The Fourier transform of the fundamental solution to the heat equation is given by $\caF\Lambda(t)(\xi) = \exp(-4\pi^2t|\xi|^2)$. Using this expression, we immediately see that $\gamma_2=1$. Moreover  $\gamma=\gamma_1=1-\beta/2$.
 Indeed, with the changes of variables $\xi\mapsto\xi-\eta$ and $\zeta=\sqrt{s}\xi$ we get
	\begin{align*}
		\int_0^t \sup_{\eta\in\Rd}\int_\Rd & |\caF\Lambda(s)(\xi+\eta)|^2\mu(d\xi)ds \\
		& = \int_0^t \sup_{\eta\in\Rd}\int_\Rd \frac{\exp(-8\pi^2|\sqrt{s}\xi|^2)}{|\xi-\eta|^{d-\beta}}d\xi ds \\
		& = \int_0^t s^{-\beta/2} ds  \sup_{\eta\in\Rd}\int_\Rd \frac{\exp(-8\pi^2|\zeta|^2)}{|\zeta-\eta|^{d-\beta}}d\zeta \\
		& = Ct^{1-\beta/2},
	\end{align*}
because the integral can be shown to be finite. Therefore, $\bar\gamma$ in Theorem \ref{thm:existencedensity} is equal to $2$, which implies the claim.

In contrast with \cite{mcms}, with the method of this article, the density for the solution to the heat equation in any spatial dimension $d\ge 1$ is proved under weaker conditions on $\sigma$ and $b$ (no differentiability is required). 
\medskip

\noindent{\bf Acknowledgment.} The authors thank Arnaud Debussche for useful discussions.


\end{document}